\documentclass[graybox]{svmult}

\usepackage{makeidx}         
\usepackage{graphicx}        
\usepackage{multicol}        
\usepackage[bottom]{footmisc}

\usepackage{newtxtext}       %
\usepackage[varvw]{newtxmath}       

\usepackage{subfigure}

\makeindex             

\usepackage{tikz,color}
\newcommand{\ep}{\varepsilon}  
\renewcommand{\d}{\mathrm{d}}  
\newcommand{\R}{\mathbb{R}}  
\renewcommand{\AA}{A}  
\newcommand{\pp}{p}
\newcommand{\vmean}{\mathbf{M}}  
\newcommand{\smean}{\mathbf{m}}  
\renewcommand{\div}{\mathrm{div}} 

\begin{document}

\title*{PDE methods for extracting normal vector fields and distance functions of shapes}
\author{Takahiro Hasebe, Jun Masamune,  Hiroshi Teramoto and Takayuki Yamada}
\institute{Takahiro Hasebe \at Department of Mathematics, Faculty of Science, Hokkaido University, 
Kita 10, Nishi 8, Kita-Ku, Sapporo, Hokkaido, 060-0810, Japan. \email{thasebe@math.sci.hokudai.ac.jp} 
\and Jun Masamune \at Graduate School of Science, Tohoku University, Sendai 980-8579, Japan. 
\email{jun.masamune.c3@tohoku.ac.jp} 
\and Hiroshi Teramoto \at Department of Mathematics, Kansai University, 3-3-35 Yamate-cho, Suita, Osaka, 564-8680, Japan. 
\email{teramoto@kansai-u.ac.jp} 
\and Takayuki Yamada \at Department of Strategic Studies, Institute of Engineering Innovation, Graduate School of Engineering, The University of Tokyo,
Yayoi 2--11--16, Bunkyo-ku, 113--8656, Tokyo, Japan. \email{t.yamada@mech.t.u-tokyo.ac.jp} }
%
%
\maketitle

\abstract*{Partial differential equations can be used for extracting geometric features of shapes. This article summarizes recent methods to extract the normal vector field from an elliptic equation proposed by Yamada and from the heat equation, and also a method to extract the (signed) distance function from an elliptic equation that generalizes Varadhan's in 1967.}

\abstract{Partial differential equations can be used for extracting geometric features of shapes. This article summarizes recent methods to extract the normal vector field from an elliptic equation proposed by Yamada and from the heat equation, and also a method to extract the (signed) distance function from an elliptic equation that generalizes Varadhan's in 1967.}

\section{Introduction}
\label{sec:1}

Extracting geometric features of shapes from partial differential equations (PDEs) has been developed in the literature, especially in the context of image analysis and engineering, see e.g.\ \cite{BF,crane2013geodesics,han2004cortical,yamada2019geometric,yamada2019thickness} and references therein. In particular, a construction of signed distance functions using PDEs is broadly useful in topology optimization \cite{AJT04}. 

Motivated by the above situation and inspired by physical intuitions, we developed PDE methods to extract the normal vector fields of sets \cite{M1} and the (signed) distance functions of sets \cite{M}.  
The present paper summarizes these results with some progress, including mathematical proofs or a sketch of proofs. Notice that ``thickness'' of shapes is another important geometric feature. Thickness is discussed in \cite{NY,yamada2019geometric,yamada2019thickness} and is omitted in the present article.  

The results are briefly described here; further details are given in respective sections. In Section \ref{sec2}, we consider the following elliptic equation  proposed by Yamada \cite{yamada2019geometric,yamada2019thickness}  for $\R^N$-valued functions $s=s_a$ defined in a bounded open subset $\Omega \subset \R^N$
\begin{equation}\label{eq:Yamada}
\begin{cases}
-\div \left(a \nabla s - {\rm Id} \chi_\AA \right)+(1-\chi_\AA)s=0&\text{in} \quad \Omega \\
s =0 & \text{on} \quad \partial\Omega, 
\end{cases}
\end{equation}
where $a>0$ is a parameter and $\chi_\AA$ is the characteristic function over a given set $\AA\subseteq \Omega$ of smooth boundary. 
It is  shown that $s_a$ converges to the normal vector field on $\partial \AA$ as $a\to0^+$ in a suitable sense, see Fig.~\ref{figsa}. 

\begin{figure}[ht]
\centering
\subfigure[$a=10^{-1}$]{\includegraphics[width=0.45\textwidth]{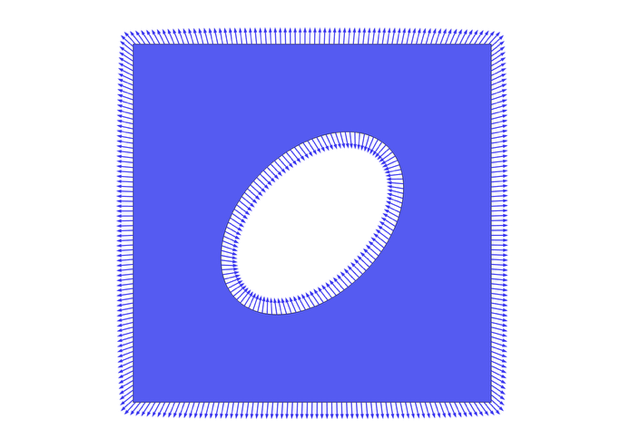}}
\subfigure[$a=10^{-3}$]{\includegraphics[width=0.45\textwidth]{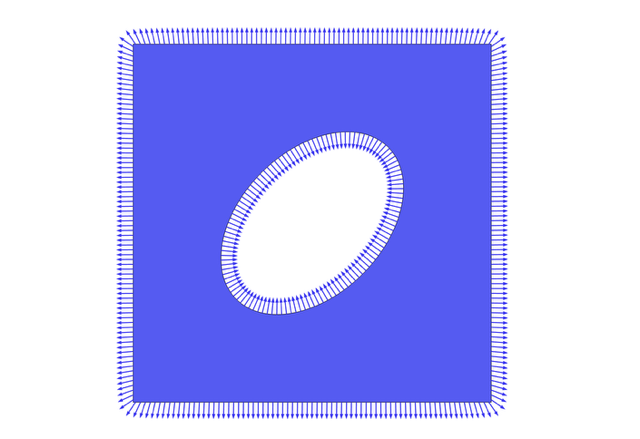}}
\subfigure[$a=10^{-4}$]{\includegraphics[width=0.45\textwidth]{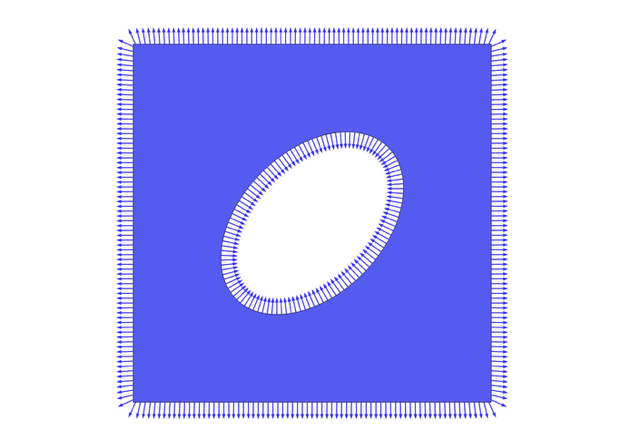}}
\subfigure[$a=10^{-6}$]{\includegraphics[width=0.45\textwidth]{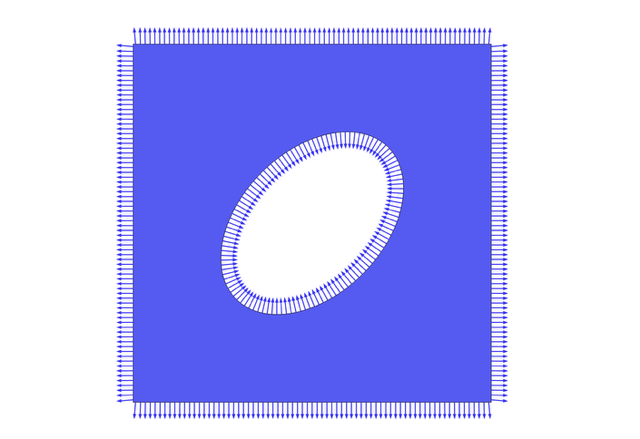}}
\caption{$\frac{s_a}{\|s_a\|}$ on the boundary of $\AA$  for different $a$'s. The total domain $\Omega$ is suitably chosen and is omitted in the figures.} \label{figsa}
\end{figure}

 We also discuss a slightly different equation
 \begin{equation}\label{eq:Yamada2}
 \begin{cases}
-\div \left(a \nabla \tilde{s} - {\rm Id} \chi_\AA \right) + \tilde{s}=0 & \text{in} \quad \Omega, \\
\tilde{s}=0 & \text{on} \quad \partial \Omega, 
\end{cases} 
\end{equation}
from which a similar conclusion follows. 
 
Section \ref{sec3} analyzes a solution $h(t,x)$ to the Cauchy problem of the heat equation
 \begin{equation}\label{eq:heat}
 \begin{cases}
 \partial_t h (t,x) = \Delta h (t,x), \\
 h(0,x) = \chi_\AA(x), \qquad\quad (t,x) \in (0,\infty)\times \R^N. 
 \end{cases}
 \end{equation}
In Subsection \ref{sec3.1}, it is shown that $-\sqrt{4\pi t}\nabla_x h(t,x)$ converges to the normal vector of $\partial \AA$ at $x\in \partial\AA$ as $t\to0^+$ if $\partial \AA$ is smooth. A rate of convergence is estimated.  Singular boundary points are then treated in Subsection \ref{sec3.2}.   Actually, the proof works for more general functions $h$ that are written as the convolution of $\chi_\AA$ and a rotationally invariant kernel, so that we give results in such a general setting.

In Section \ref{sec4}, we analyze a solution $u=u_a$ to the elliptic equation 
\begin{equation}\label{eq:sd0}
\begin{cases}
-a\Delta u+u = f & \text{in} \quad \Omega,\\
u=g&\text{on}\quad \partial \Omega,      
\end{cases}
\end{equation}
where $a>0$ is a parameter. It is shown that $-\sqrt{a} \log u_a(x)$ converges to the distance function $d(x,\partial \AA)$ on $\AA$ as $a\to0^+$, where $\AA:= \Omega \setminus \overline{\{x \in \Omega: f(x)>0\}}$. 
In the particular case where $f = C^* \chi_{\Omega \setminus \overline{\AA}}$ with any constant $C^* \ge \sup_{x\in \partial\Omega}g(x)$, we show that the function  
\begin{align}\label{eq:sdist}
    U_a(x):=
\begin{cases}
\sqrt{a}\log u_a(x),\quad &x \in \overline{\AA}, \\ 
-\sqrt{a}\log ( C^* -u_a(x)),\quad &x\in \Omega \setminus \overline{\AA} 
\end{cases}
\end{align}
converges to the signed distance function of $\partial \AA$
on a subdomain $\Omega^*$ such that $\overline{\AA} \subset \Omega^* \subset\Omega$. 
Rate of convergence is also provided.

\section{Normal vector field: elliptic equations proposed by Yamada} \label{sec2}

Let $\AA$ and $\Omega$ be open bounded subsets of $\R^N~(N\ge2)$ such that the closure  $\overline{\AA}$ is contained in $\Omega$.  The boundary of $\AA$ is assumed to be Lipschitz in the sense of \cite[Definition 9.57]{Leo}, i.e., for each $x \in \partial\AA$ there is a neighborhood $U \subset \R^N$ of $x$ such that $\partial\AA \cap U$ can be represented as the graph of a Lipschitz function in some orthogonal basis. Lipschitz continuity implies the differentiability almost everywhere and allows us to define the surface measure of $\partial \AA$ denoted by $\sigma$; see e.g.~\cite[Example 9.56]{Leo}. Then we can define the normal vector of the surface $\partial \AA$, oriented outwards from $\AA$ and defined $\sigma$-almost everywhere. The normal vector at $x \in \partial \AA$ is denoted by $n(x)$. 

The standard inner product in $\R^N$ is signified by $(\cdot,\cdot)$ and the Hilbert-Schmidt inner product  in the matrix space $M_N(\R)$ is denoted by $(A: B):=\text{Tr}(A^{T}B)$. The inner product in the $\R^N$-valued Sobolev space $H_0^1(\Omega)^N$ is given by  
\[
\langle \varphi, \psi \rangle_{H_0^1(\Omega)^N} :=  \int_\Omega (\nabla \varphi: \nabla \psi)\, \d x= \sum_{i,j=1}^N  \langle \partial_{x_i}\varphi_j,  \partial_{x_i}\psi_j\rangle_{L^2(\Omega)}
\]
for $\varphi=(\varphi_1,\varphi_2,\dots, \varphi_N), \psi=(\psi_1,\psi_2,\dots, \psi_N) \in H_0^1(\Omega)^N$.

Equation \eqref{eq:Yamada} is formulated in the weak form 
\begin{equation}\label{eq:yamada}
a \int_\Omega (\nabla s: \nabla \varphi)\, \d x  + \int_{\Omega\setminus \AA}(s,\varphi) \, \d x = \int_{\partial \AA} (n,\varphi) \,\d\sigma, \qquad \varphi \in H_0^1(\Omega)^N,  
\end{equation}
where $a$ is a positive parameter.  
On the right hand side, the function $\varphi$ on the surface $\partial \AA$, sometimes denoted by $\varphi|_{\partial\AA}$,  is defined via the trace theory, see e.g.\ \cite[Theorem 18.1]{Leo}.  
In particular, the inequality 
\begin{equation} \label{eq:Trace}
\|\varphi \|_{L^2(\partial \AA)^N} \le C_1 \|\varphi\|_{H_0^1(\Omega)^N}
\end{equation}
holds for some constant $C_1>0$ independent of $a$ and $\varphi$, so that the surface integral in \eqref{eq:yamada} is a bounded linear functional on $H_0^1(\Omega)^N$. By Lax-Milgram's theorem (see e.g.\ \cite[Corollary 5.8]{B11}), Equation \eqref{eq:yamada} has a unique solution $s=s_a\in H_0^1(\Omega)^N$. 

\begin{theorem} [\cite{M1}] \label{thm1}The following holds:  
	\begin{equation}\label{eq:normal}
	\lim_{a\to0^+} \int_{\Omega \setminus \AA}(s_a, \varphi)  \, \d x = \int_{\partial \AA} (n,\varphi) \, \d\sigma, \qquad  \varphi\in H_0^1(\Omega)^N. 
	\end{equation}
	Moreover, the strong convergence $a s_a \to 0$ in $H_0^1(\Omega)^N$   holds. 
\end{theorem}
\begin{remark}\label{rem1}
	Assertion \eqref{eq:normal} means that the vector-valued function $s_a |_{\Omega\setminus \AA}$ converges to the vector-valued measure $n(x)\delta_{\partial \AA}(\d x)$ in the dual space of $\{\varphi|_{\Omega \setminus \AA}: \varphi \in H_0^1(\Omega)^N \}$, where $\delta_{\partial \AA}$ stands for the finite Borel measure on $\Omega \setminus \AA$
	\begin{equation}\label{eq:delta}
	\delta_{\partial\AA}(B) := \text{Area}(B\cap \partial \AA) =  \int_{B\cap \partial \AA} \d \sigma, 
	\end{equation}
	which is singular to the Lebesgue measure on $\Omega \setminus \AA$.  
\end{remark}

\begin{proof}[Proof of Theorem \ref{thm1}] By taking the test function $\varphi = s_a$, the weak form yields 
\begin{equation}\label{eq:s1}
a \|s_a \|_{H_0^1(\Omega)^N}^2 \le a \|s_a \|_{H_0^1(\Omega)^N}^2 +  \|s_a \|_{L^2(\Omega\setminus\AA)^N}^2 = \int_{\partial \AA} (n, s_a)\, \d \sigma. 
\end{equation}
By the Schwarz inequality and \eqref{eq:Trace}, the last surface integral is further bounded by 
\begin{align}
\left|\int_{\partial \AA} (n, s_a)\, \d \sigma  \right|   \label{eq:s2}
&\le \|n\|_{L^2(\partial \AA)^N} \|s_a\|_{L^2 (\partial \AA)^N} = \sqrt{\text{Area}(\partial \AA)}\,\|s_a\|_{L^2 (\partial \AA)^N} \\
&\le C_2 \|s_a\|_{H_0^1(\Omega)^N}, \notag 
\end{align}
where $C_2 :=\sqrt{\text{Area}(\partial \AA)} \, C_1$ is independent of $a>0$.  Combining \eqref{eq:s1} and \eqref{eq:s2} together yields  
\begin{equation} \label{eq:uniform_bound}
 \|a s_a \|_{H_0^1(\Omega)^N} \le C_2.   
\end{equation}
Moreover,  \eqref{eq:uniform_bound} and \eqref{eq:s2} imply that the surface integral is bounded by $C_2^2 /a$, so that \eqref{eq:s1} (multiplied by $a$) gives
\begin{equation} \label{eq:s3}
 \| a s_a \|_{H_0^1(\Omega)^N}^2 +  \|\sqrt{a} s_a \|_{L^2(\Omega\setminus\AA)^N}^2 \le C_2^2. 
\end{equation}
 Assertion \eqref{eq:normal} is equivalent to that $as_a$ weakly converges to $0$ in $H_0^1(\Omega)^N$ as $a\to0^+$.  By the uniform boundedness \eqref{eq:uniform_bound}, we can extract a subsequence of $\{a s_a\}_{0<a<1}$ as $a\to0^+$,   simply denoted by $\{a s_a\}$, that converges weakly to some $w$ in $H_0^1(\Omega)^N$, see e.g.\ \cite[Theorem 3.18]{B11}. This implies that the same subsequence $\{a s_a\}$ also weakly converges in $L^2(\Omega)^N$ to $w$ because the continuous embedding $L^2(\Omega)^N \to H_0^1(\Omega)^N$ induces the dual continuous embedding $(H_0^1(\Omega)^N)^* \to (L^2(\Omega)^N)^*$.  From \eqref{eq:s3}, the subsequence $\{\sqrt{a} s_a\}$ is bounded in $L^2(\Omega \setminus \AA)^N$ and so $\{a s_a\}$ converges to 0 in $L^2(\Omega \setminus \AA)^N$. This shows $w|_{\Omega\setminus \AA} =0$ a.e. According to the definition of $w|_{\partial \AA}$ via the trace theory, $w|_{\partial \AA}=0$ a.e. too. By taking $\varphi=w$ in the weak form \eqref{eq:yamada}, we get 
\[
\int_\Omega (a \nabla s_a: \nabla w)\, \d x = 0. 
\]
We further take the limit $a\to0^+$ along the subsequence we have selected above to obtain $\|w\|_{H_0^1(\Omega)^N}=0$, i.e., $w=0$ a.e. By the standard  compactness argument, without going to subsequences, the weak convergence $as_a \rightharpoonup 0$ in $H_0^1(\Omega)^N$ fully follows. 

For the notational clarity we set $T\colon H_0^1(\Omega)^N \to L^2(\partial \AA)^N$ to be the bounded linear operator $\varphi \mapsto \varphi|_{\partial\AA}$. The already established fact $a s_a \rightharpoonup 0$ in $H_0^1(\Omega)^N$ implies the weak convergence $T(as_a) \rightharpoonup 0$ in $L^2(\partial \AA)^N$ because $f \circ T \in (H_0^1(\Omega)^N)^*$ for all $f \in (L^2(\partial \AA)^N)^*$.  Hence,  
\begin{equation}\label{eq:tr}
a \int_{\partial \AA}(n,s_a) \,\d \sigma = \langle n, T(as_a)\rangle_{L^2(\partial \AA)^N} \to 0 \qquad \text{as} \qquad a\to0^+. 
\end{equation}
Selecting the test function $\varphi = as_a$ in the weak form gives 
\[
\|a s_a\|_{H_0^1(\Omega)^N}^2 + \|\sqrt{a}s_a\|_{L^2(\Omega\setminus \AA)^N}^2 = a \int_{\partial \AA}(n,s_a) \,\d \sigma, 
\]
which, together with \eqref{eq:tr}, establishes the strong convergence as desired. 
\end{proof}

Next we discuss Equation \eqref{eq:Yamada2}. Its weak form is 
\begin{equation*}\label{eq:yamada2}
a \int_\Omega (\nabla \tilde{s}: \nabla \varphi)\, \d x  + \int_{\Omega}(\tilde{s},\varphi) \, \d x = \int_{\partial \AA} (n,\varphi) \,\d\sigma, \qquad \varphi \in H_0^1(\Omega)^N.   
\end{equation*}
As before, a unique solution $\tilde{s}=\tilde{s}_a \in H_0^1(\Omega)^N$ exists. 

\begin{theorem}\label{thm2} It holds that 
\begin{equation}\label{eq:normal2}
	\lim_{a\to0^+} \int_{\Omega}(\tilde{s}_a, \varphi)  \, \d x = \int_{\partial \AA} (n,\varphi) \, \d\sigma, \qquad  \varphi\in H_0^1(\Omega)^N 
	\end{equation}
and that $a \tilde{s}_a \to 0$ strongly in $H_0^1(\Omega)^N$. 
\end{theorem}
\begin{remark} 
Similarly to Remark \ref{rem1}, assertion \eqref{eq:normal2} implies that $\tilde{s}_a$ converges to the $\R^N$-valued finite Borel measure $n(x)\delta_{\partial \AA}(\d x)$ on $\Omega$, where $\delta_{\partial\AA}$ is defined by the same \eqref{eq:delta} as before but now is regarded as a measure on the larger domain $\Omega$. 
The convergence primarily holds in the dual space $(H_0^1(\Omega))^*$, so that it particularly holds in the Schwartz distribution space $\mathcal{D}'(\Omega)$.  \end{remark}
\begin{proof}[Proof of Theorem \ref{thm2}]
The proof is very similar to Theorem \ref{thm1}. The main difference is that inequality \eqref{eq:s3} is improved into 
\begin{equation*} \label{eq:ss3}
 \| a \tilde{s}_a \|_{H_0^1(\Omega)^N}^2 +  \|\sqrt{a} \tilde{s}_a \|_{L^2(\Omega)^N}^2 \le C_2^2. 
\end{equation*}
Therefore, the family $\{\sqrt{a}\tilde{s}_a\}_{0<a<1}$ is uniformly bounded in $L^2(\Omega)^N$, which actually makes the arguments for proving $a\tilde{s}_a \rightharpoonup 0$ in $H_0^1(\Omega)^N$ slightly simpler than the case of $as_a$.  
\end{proof}

\begin{note}
Theorem \ref{thm2} and the strong convergence $as_a \to 0 $ in Theorem \ref{thm1} are new results that were not included in \cite{M1}. 
\end{note}

\section{Normal vector field: the heat equation method}\label{sec3}
Let $N\ge2$ and $\AA$ an open subset of $\R^N$. In this section, $\kappa\colon [0,\infty) \to \R$ is  a fixed $C^1$ function such that 
\begin{description}[($\kappa1$)]
\item[($\kappa1$)]{$\displaystyle\lim_{r\to \infty} \kappa(r)=0$ and $\displaystyle \int_{\R^N} |\kappa(\|x\|^2)|\,\d x <\infty$,}  
\item[($\kappa2$)]{$\displaystyle\lim_{r\to \infty} r \kappa'(r^2)=0$ and  $\displaystyle\int_{\R^N}  \sup_{b \in K}\left( \|x + b\| \!\cdot\!|\kappa'(\|x+b \|^2)|\right)\d x<\infty$ for any compact subset $K \subset \R^N$.} 
\end{description}
Let $\kappa_t(x) := t^{-N/2} \kappa(\|x\|^2/t)$ and 
\[
h(t,x) := \int_{\R^N} \kappa_t(x-y) \chi_\AA(y)\, \d y, \qquad t>0,~ x\in \R^N.  
\]
Assumption $(\kappa2)$ allows us to calculate the gradient
\begin{equation}\label{eq:nablah}
\nabla_x h(t, x) = \int_{\R^N} \frac{2(x-y)}{t^{N/2+1}} \kappa'(\|x-y\|^2/t)\chi_\AA(y) \,\d y. 
\end{equation}
The goal of this section is to demonstrate that $\nabla_x h(t, x)$ is approximately oriented toward the normal direction on  $\partial \AA$ whenever $x$ is a smooth point of the boundary. If $x$ is a singular point of the boundary then the normal vector is not defined. In this case, we still investigate the asymptotic behavior of the gradient $\nabla_x h(t, x)$.  

In the special case where $\kappa(r) = (4\pi)^{-N/2} \exp(-r/4)$, of course $h$ is a solution to the heat equation \eqref{eq:heat}. Since most of the results hold for general kernels $\kappa$ and the proof is almost the same, we deal with general $\kappa$.

\subsection{Smooth boundary points}\label{sec3.1}

Let $\alpha \in (0,1]$. 
The boundary $\partial \AA$ is said to be a \emph{$C^{1,\alpha}$-surface at $p\in \partial\AA$} if   there are an open neighborhood $U\subseteq\R^N$ of $\pp$ and a $C^{1,\alpha}$-diffeomorphism $\varphi$ of $U$ onto an open neighborhood  $\varphi(U)\subseteq \R^N$ of $0$  such that $\varphi(\pp)=0$ and $\varphi(A\cap U) = \{x \in \varphi(U): x_N>0\}$. 
An equivalent definition is that $\partial \AA$ can be written as the graph of a $C^{1,\alpha}$-function of $N-1$ variables in a neighborhood of $\pp$, i.e., there are an orthonormal basis $\{v_1,v_2,\dots, v_N\}$ of $\R^N$, an open hyperrectangle $V=\prod_{i=1}^N(-\delta_i,\delta_i)$ and  a $C^{1,\alpha}$-function $f\colon V':=\prod_{i=1}^{N-1} (-\delta_i,\delta_i) \to (-\delta_N,\delta_N)$ such that $f(0)=0$ and 
\begin{equation*} \label{eq:loc}
\AA \cap \left\{p+ \sum_{i=1}^N x_i v_i: x \in V \right\} = \left\{p +  \sum_{i=1}^N x_i v_i: x'\in V',  f(x') < x_N < \delta_N \right\}
\end{equation*}  
with shorthand notation $x'=(x_1,x_2,\dots, x_{N-1})$. 
In case $\partial \AA$ is a $C^{1,\alpha}$-surface at $\pp$, let $n$ signify the normal vector field of $\partial\AA$ in a neighborhood of $\pp$ oriented outwards from $\AA$ as before. Naturally, $n$ is $\alpha$-H\"older continuous in a neighborhood of $\pp$.

\begin{theorem}\label{Thm:normal} Let  $\alpha\in (0,1]$. Assume that $\partial \AA$ is a $C^{1,\alpha}$-surface at $\pp\in \partial \AA$. Then 
\begin{equation}\label{eq:asymptotic_normal}
\nabla_x h(t,x) = -  t^{-\frac1{2}}n(x) (\xi + o(1)) \quad  \text{as}\quad t\to 0^+ 
\end{equation}
uniformly on $\partial \AA \cap U$ for some neighborhood $U\subseteq\R^N$ of $\pp$, where $\xi$ is the scalar 
$$
\xi= S_{N-2} \int_{0}^\infty \kappa(r^2)r^{N-2}\,\d r  
$$
and $S_{N-2}$ is the area of the unit sphere in $\R^{N-1}$. We interpret $S_0=2$.   
\end{theorem}
\begin{proof} We fix $x\in \partial \AA$ in a neighborhood of $\pp$ and perform the change of variables $y-x=\sqrt{t}z$ in \eqref{eq:nablah} to get 
\begin{equation}\label{eq:heat_nabla}
\nabla_x h(t,x) =  -\frac{2}{\sqrt{t}}\int_{\{z\in \R^N:\, x+\sqrt{t}z \in \AA\}} z \kappa'(\|z\|^2) \,\d z.
\end{equation}
We approximate the integral region with a half-space that is tangential to $\partial\AA$.  We begin with truncating the integral as
\begin{equation}\label{eq:cut}
\nabla_x h(t,x) = -\frac{2}{\sqrt{t}}\left(\int_{\{z\in\R^N:\,\|z\|\leq R,~x+\sqrt{t}z \in \AA\}} z \kappa'(\|z\|^2) \,\d z +\ep_1(t) \right),   
\end{equation}
where $R>0$ is supposed to be large and the error term $\ep_1(t)$ is bounded as
\begin{equation}\label{eq:e1}
\|\ep_1(t)\| \le \int_{\{z\in\R^n:\,\|z\|> R\}} \|z\|\!\cdot\! |\kappa'(\|z\|^2)| \,\d z. 
\end{equation}
Without loss of generality, we assume that $\AA$  in a neighborhood of $x$ can be represented as a graph $\{x+z: z_N > f(z'), z' \in V'\}$  in the standard basis, where $f$ is a  $C^{1,\alpha}$-function with $f(0)=0$ and $z= (z', z_N)$. Let $H:=\{x+ z\in \R^N: z_N > (\nabla f(0), z')\}$. Note that $\partial H$ is the tangent plane of $\partial \AA$ and it has normal vector $n(x)$ at $x$.  
Assuming $\sqrt{t}R$ is small, we estimate the volume of the symmetric difference of $B=\{z \in \R^N: x+\sqrt{t}z \in \AA, \|z\|\leq R\}$ and $B'= \{z\in\R^N: z\in H-x, \|z\|\leq R\}$. Let $g(z'):= f(z') - (\nabla f(0),z').$ Since $g(0) = \nabla g(0)=0$, we have a bound $|g(z')| \le C\|z'\|^{1+\alpha}$ in a neighborhood of $0$ for some $C>0$. The symmetric difference $B\Delta B' := (B\setminus B') \cup (B'\setminus B)$, when multiplied by the factor $\sqrt{t}$ and shifted by $x$, is the region shown in Fig.~\ref{fig1} surrounded by the graph of $x+f$, the hyperplane $\partial H$ and the sphere of center $x$ and radius $\sqrt{t}R$, so that its volume is estimated as     
\begin{align}
 |B\Delta B'| &= \frac{ |(\sqrt{t}B)\Delta (\sqrt{t}B')|}{t^{N/2}} \le \frac1{t^{N/2}}\int_{\|z'\|\le \sqrt{t}R} |g(z')| \,\d z'  \le \frac{C S_{N-2}}{N+\alpha}  t^\frac{\alpha}{2}R^{N+\alpha}.    \label{eq:vol}
\end{align}
\begin{figure}[t]
\sidecaption[t]
\begin{tikzpicture}[scale =1.1]
\draw[domain=-2.3:2.3] plot(\x, 0.2*\x*\x);
\draw (-2.4,0) -- (2.4,0);
\draw node at (1.3,1) {$\AA$};  
\draw node at (1.8,0.3) {$H$};  
\draw node at (0.05,0.02) [below left] {$x$};  
\path[fill, opacity=0.8] plot[domain=-0.981281:0.981281] (\x, 0.2*\x*\x)  arc [start angle = 11.1, end angle = 168.9, radius = 1] plot[domain=-0.89:0] (\x, 0.3*\x*\x-0.3*\x);
\path[fill, opacity=0.3]  plot[domain=0:1] (\x, 0) arc [start angle = 0, end angle = 11.1, radius = 1] plot[domain=0.981281:0] (\x, 0.2*\x*\x);
\path[fill, opacity=0.3]  plot[domain=0:-1] (\x, 0)  arc [start angle = 180, end angle = 168.9, radius = 1]  plot[domain=-0.981281:0] (\x, 0.2*\x*\x);
\draw[->] (0,0) -- (0,-0.7); 
\draw node at (0,-0.7) [below] {$n(x)$}; 
\draw[<->] (0.02,-0.1) -- (1,-0.1); 
\draw node at (0.5,-0.04) [below] {\small$\sqrt{t}R$};  
\end{tikzpicture}
\caption{Approximating $\AA$ with $H$ at $x$. The black region stands for $x+\sqrt{t}B$ and the gray region for $x+(\sqrt{t}B)\Delta (\sqrt{t}B')$.}\label{fig1}
\end{figure}
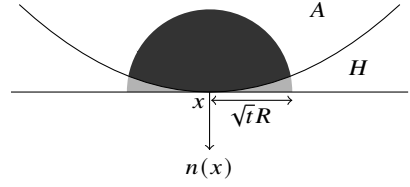
Taking $R=R(t)$ as a function of $t$ such that $R(t)\to\infty$ and $t^\frac{\alpha}{2}R^{N+\alpha}\to0$ as $t\to0^+$ (e.g., $R(t) =\log\frac1{t}$) yields  the following estimate for the integral in \eqref{eq:cut}: 
\begin{eqnarray}
 \int_{B} z \kappa'(\|z\|^2) \,\d z  \notag
&=& \left( \int_{B'}  + \int_{B\setminus B'} - \int_{B'\setminus B}\right) z \kappa'(\|z\|^2) \,\d z   \notag \\
&=& \int_{B'} z \kappa'(\|z\|^2) \,\d z  + \ep_2(t) \label{approx1}\\
&=&\int_{\|w\|\leq R,~w_N > 0} \left( \sum_{i=1}^{N-1} w_i \ell_i -w_N n(x) \right) \kappa'(\|w\|^2) \,\d w+\ep_2(t) \notag \\
&=& -n(x)\int_{\{w\in \R^N:\, w_N>0\}} w_N \kappa'(\|w\|^2) \,\d w +\ep_2(t) +\ep_3(t).  \label{approx2}
\end{eqnarray}
In the course of calculations above,  the change of variables $z =  w_1 \ell_1 + \cdots +w_{N-1} \ell_{N-1} -w_N n(x)$ is employed, where $\{\ell_1,\dots, \ell_{N-1}, n(x)\}$ is an orthonormal basis. Due to the oddness of integrands, the terms except $-w_N n(x)$ all vanished. The error terms $\ep_2(t)$ and $\ep_3(t)$ can be estimated in the ways 
\begin{align}
\|\ep_2(t)\| &\le \left(\sup_{z \in \R^N  }\|z\| \!\cdot\!  |\kappa'(\|z\|^2)|\right) |B\Delta B'|  =O( t^\frac{\alpha}{2}R^{N+\alpha}), \label{eq:E2} \\ 
\|\ep_3(t)\| &\le \int_{\{z\in\R^N:\,\|z\|>R\}} \|z\|\!\cdot\!  |\kappa'(\|z\|^2)|\,\d z =o(1). \label{eq:E3}
\end{align}
Notice that $\ep_1(t)$ and $\ep_3(t)$ are bounded by the same integral. 
For the last integral in \eqref{approx2}, because $w_N \kappa'(\|w\|^2) = \partial_{w_N} [\kappa(\|w\|^2)/2]$ we obtain with notation $w=(w', w_N)$
\begin{align}
\int_{\{w \in \R^N :\, w_N>0\}} w_N \kappa'(\|w\|^2) \,\d w
&= -\frac1{2}\int_{\R^{N-1}} \kappa(\|w'\|^2) \,\d w'  \notag \\
&= -\frac{S_{N-2}}{2} \int_{0}^\infty \kappa(r^2)r^{N-2}\,\d r. \label{eq:calc}
\end{align}
  Putting \eqref{eq:cut}, \eqref{approx2}, \eqref{eq:calc} together gives the conclusion \eqref{eq:asymptotic_normal}, where the error term $o(1)$ is  bounded by $\|\ep_1(t)\| +\|\ep_2(t)\| +\|\ep_3(t)\|$.  Because the constant $C$ in \eqref{eq:vol} is essentially the H\"older constant for $\nabla f$,  the estimate \eqref{eq:vol} and hence \eqref{eq:E2} are uniform over a neighborhood of $\pp\in \partial\AA$. 
\end{proof}

One can also obtain the rate of convergence. Since the error term $o(1)$ in \eqref{eq:asymptotic_normal} is bounded by $\|\ep_1(t)\| +\|\ep_2(t)\| +\|\ep_3(t)\|$, making the latter quantity smaller gives a better rate. This is achieved typically when the three terms are comparable.  Here is an example from the heat equation. 
\begin{example}[Heat equation]
Let $D>0$ and $\kappa(r) := (4\pi D)^{-N/2}\exp(-\frac{r}{4D})$. Then $h$ satisfies $h_t=D\Delta h$. In this case $\xi=(4\pi D)^{-\frac{1}{2}}$ and we can verify 
\begin{equation}\label{eq:asymptotic_heat}
\nabla_x h(t,x) = \frac{-n(x)}{\sqrt{4\pi Dt}} \left(1+O\left(t^{\frac{\alpha}{2}}\log^{\frac{N+\alpha}{2}}\frac1{t}\right) \right) \quad  \text{as}\quad t\to 0^+. 
\end{equation}
To see this, we set $R=\sqrt{3D\log (1/t)}$ and proceed as 
\begin{eqnarray*}
\|\ep_1(t)\|, \|\ep_3(t)\| &\le & \int_{\{z\in\R^N:\, \|z\|>R\}} \|z\| e^{-\frac{\|z\|^2}{4D}}\frac{\d z}{(4\pi D)^{N/2}}  \\
&\le& e^{-\frac{0.9}{4D}R^2} \underbrace{\int_{\R^N} \|z\| e^{-\frac{0.1}{4D}\|z\|^2}\frac{\d z}{(4\pi D)^{N/2}}}_{=:\beta} =  \beta t^{0.675}. 
\end{eqnarray*}
Observe from \eqref{eq:E2} that $\|\ep_2(t)\| = O(t^{\frac{\alpha}{2}}R^{N+\alpha})=O(t^{\frac{\alpha}{2}}\log^{\frac{N+\alpha}{2}}(1/t))$, which  converges more slowly to $0$ than the bounds obtained for $\ep_1(t),\ep_3(t)$, so that \eqref{eq:asymptotic_heat} holds. 
\end{example}

\subsection{Singular boundary points}\label{sec3.2}

We analyze $\nabla_x h(t,x)$ as $t\to0^+$ at singular points $x$ in $\partial\AA$.

\subsubsection{Approximating cones}

The crucial step of Theorem \ref{Thm:normal}  was \eqref{approx1}, where the integral region for $z$ was approximated by the half-space $H-x$. From this perspective Theorem \ref{Thm:normal} can be generalized to boundary singular points. A  subset $C$ of $\R^N$ is called a \emph{cone at $\pp\in \R^N$} if for every $z \in C -\pp$ the half-line $\{t z: t>0\}$ is contained in $C-\pp$. 

\begin{theorem}\label{Thm:gen} Let $N\ge2$ and $\pp\in \partial\AA$. Suppose that there is an open cone $C$ at $\pp$ (called an approximating cone of $\AA$ at $\pp$) that well approximates $\AA$ in small neighborhoods of $\pp$, in the sense that there is a function $(0,1) \ni t\mapsto R(t) \in (0,\infty)$ such that 
\begin{align}
&\lim_{t\to0^+}R(t)=\infty, \quad \lim_{t\to0^+}\sqrt{t}R(t)=0 \quad \text{and}  \quad \lim_{t\to0^+}t^{-\frac{N}{2}}|U_t \Delta V_t| =0,  \label{eq:UV}
\end{align}
where 
\begin{align*}
U_t &:=\{w\in  \AA-\pp: \|w\|\le \sqrt{t}R(t) \}, \\
V_t &:=\{w \in C - \pp: \|w\| \le \sqrt{t}R(t) \}. 
\end{align*}
Then we have 
\begin{equation} \label{eq:normal3}
\nabla_x h(t,\pp) = -  \frac{2}{\sqrt{t}}  \left(\int_{C-\pp} z \kappa'(\|z  \|^2) \,\d z + o(1)\right) \quad  \text{as}\quad t\to 0^+. 
\end{equation}
\end{theorem}
\begin{proof}
The proof is just to extract the corresponding ideas from the proof of Theorem  \ref{Thm:normal}, see also Example \ref{ex:normal} below. Estimate \eqref{eq:cut} still holds with the same proof. As a consequence of \eqref{eq:UV}, we have 
\begin{align*}
 \int_{\{z\in \R^N:\, \|z\|\leq R(t),\, \pp+\sqrt{t}z \in \AA\}} z \kappa'(\|z\|^2) \,\d z 
 &=   \int_{U_t/\sqrt{t}} z \kappa'(\|z\|^2) \,\d z  \\
&=  \int_{V_t/\sqrt{t}} z \kappa'(\|z\|^2) \,\d z +o(1)  \\
&=  \int_{\{z\in C - \pp:\, \|z\|\leq R(t)\}} z \kappa'(\|z\|^2) \,\d z +o(1) \\
&= \int_{C - \pp} z \kappa'(\|z\|^2) \,\d z + o(1).  
\end{align*}
The third equality above is based on the fact that $z \in C-\pp$ is equivalent to $\sqrt{t}z \in C-\pp$. 
The last equality uses an estimate similar to \eqref{eq:e1}. 
\end{proof}

\begin{remark} If the main term $v(p):=\int_{C-\pp} z \kappa'(\|z  \|^2) \,\d z$ is nonzero, then $n(p):=v(p)/\|v(p)\|$ can naturally be selected as a definition of normal vector at $p$.  
\end{remark}

\begin{remark} A closely related notion is \emph{tangent cone} (see e.g.~\cite{Mo}). 
The closure of an approximating cone coincides with a tangent cone in some examples. 
\end{remark}

\begin{example}[Relationships to Theorem \ref{Thm:normal}] \label{ex:normal} If $\partial \AA$ is a $C^{1,\alpha}$-surface at $\pp$ then we can take $C$ to be the half-space separated by the tangent plane of $\partial\AA$ at $\pp$ and  containing the half-line $\{\pp - t n(\pp): t>0\}$. The sets $U_t$ and $V_t$ coincide with $\sqrt{t}B$ and $\sqrt{t}B'$ in the proof of Theorem \ref{Thm:normal}, respectively. 
\end{example}

\begin{example}[Cusps]
Theorem \ref{Thm:gen} can be applied to some cusps.  Let $ \AA \subset \R^2$ be an open set such that 
\[
\AA \cap \{x \in \R^2: \|x\|<\ep\} = \{(x_1,x_2):  x_1>0, ~0<x_2 < x_1^{1+\alpha}\}  \cap \{x: \|x\|<\ep\}
\] 
for some $\ep,\alpha>0$. At the origin $p=(0,0)$, one can easily see that $t^{-1}|U_t| =O(t^{\frac{\alpha}{2}}R^{2+\alpha})$, so that \eqref{eq:UV} holds for the empty cone $C=\emptyset$ and for the function $R(t) = \log\frac1{t}$. The main term $\int_{C} z \kappa'(\|z  \|^2) \,\d z$ vanishes. 

\end{example}

\begin{example}[Rectangular cuboid]  \label{ex_rectangular}
Let $\alpha,\beta,\gamma>0$ and $\AA = (0,\alpha) \times (0,\beta) \times (0,\gamma)$. We fix any function $t\mapsto R(t)>0$  such that $\lim_{t\to0^+}R(t)=\infty, \lim_{t\to0^+}\sqrt{t}R(t)=0$. 

Let $\pp =(m,0,0) \in \partial \AA$ with $0<m<\alpha$. Then $C :=\{(x,y,z) \in \R^3: y,z>0\}$ is a desired cone at $\pp$; indeed, $U_t = V_t$ for sufficiently small $t>0$, so that condition \eqref{eq:UV} obviously holds. Considering formula \eqref{eq:normal} and the symmetry, $\sqrt{t}\nabla_x h(t,\pp)$ converges to the vector $(0,-1,-1)$ up to a positive constant factor. 

On the other hand, let $\pp=(0,0,0)\in \partial \AA$. Then the cone $C :=\{(x,y,z) \in \R^3: x, y,z>0\}$ at $\pp$ satisfies the assumptions of Theorem \ref{Thm:gen}. Again by symmetry, $\sqrt{t}\nabla_x h(t,\pp)$ converges to the vector $(-1,-1,-1)$ up to a positive constant factor. 
\end{example}

\subsubsection{Corners} 
We consider a class of singular points that admit approximating cones. 
Let $\pp \in \partial\AA$ and $0<\alpha \le 1$. Suppose that there is a  neighborhood $U \subseteq \R^N$ of $\pp$, $C^{1,\alpha}$-functions $g_i\colon U \to \R, i=1,2,\dots, k$, and a partition $\{i: 1\le i \le k\} = I_1 \cup I_2 \cup \cdots \cup I_m$ ($I_1,I_2,\dots, I_m$ being disjoint) such that 
$g_i(\pp)=0$ for all $1\le i \le k$ and 
\begin{equation}\label{eq:MF}
 \AA \cap U = \bigcup_{j=1}^m \{x \in U: g_i(x)>0~\text{for all $i\in I_j$}\}. 
\end{equation}
We call $\pp$ a \emph{$C^{1,\alpha}$-corner of $\AA$} if the following conditions are fulfilled: 
\begin{description}[($\kappa1$)]
\item[(C1)]{there exists a vector $\ell \in \R^N$ such that $(\nabla g_i(\pp),\ell)>0$ for all $1\le i \le k$,}  
\item[(C2)]{for any $i\ne j$ the surfaces $\{x: g_i(x)=0\}$ and $\{x: g_j(x)=0\}$ intersect transversally in a neighborhood of $\pp$, or equivalently, $\nabla g_i(\pp)$ and $\nabla g_j(\pp)$ are linearly independent. } 
\end{description}
Condition (C1) and the implicit function theorem imply that all surfaces $\{x: g_i(x)=0\}$ can locally be expressed as the graphs of $C^{1,\alpha}$-functions defined on a common hyperplane $H_\ell$ that is perpendicular to $\ell$. Condition (C2) guarantees that the subset $\{x: g_i(x) =0, g_j(x)=0\}$ in a neighborhood of $\pp$ is an $(N-2)$-dimensional $C^{1,\alpha}$-submanifold of $\R^N$. 

\begin{remark}
The term ``corner'' is defined differently in the literature, e.g.~in \cite{Mi}.  
\end{remark}

\begin{remark} In case $m=1$, condition (C1) is called the Mangasarian-Fromovitz constraint qualification that is considered in the context of nonlinear programing (see e.g.~\cite{JJT,P} and references therein).  
\end{remark}



\begin{theorem}\label{prop:corner}  Let $N\ge2$ and $\pp$ be a $C^{1,\alpha}$-corner of $\AA$. Suppose that $\AA$ is written as \eqref{eq:MF} in a neighborhood $U$ of $\pp$. Then 
\begin{equation}\label{eq:C}
C= \bigcup_{j=1}^m \{x \in \R^N: (\nabla g_i (\pp),x-\pp)>0~\text{for all $i \in I_j$}\}
\end{equation}
 is an approximating cone of $\AA$ at $\pp$ and hence formula \eqref{eq:normal3} holds. 
\end{theorem}
\begin{proof} 
Let $R=R(t)>0$ be a function such that $\sqrt{t}R(t)\to0$ as $t\to0^+$. Let $U_t=\{z \in \AA-p: \|z\|\leq \sqrt{t}R\}$ and $V_t=\{z\in C - \pp: \|z\|\le \sqrt{t}R\}$.

\vspace{2mm}
\noindent
\textbf{Two dimensions.} Conditions (C1) and (C2) imply that at most two surfaces $\{g_i=0\}$ suffice to describe $\AA$ locally. This means that, for some $U = p+ (-\delta,\delta)\times (-\ep,\ep)$ in a suitable orthonormal basis, $\AA \cap U$ in \eqref{eq:MF} is of the form 
\begin{align} 
&\pp + \{(x,y): -\delta < x \le 0,  f_1(x) < y <\ep \} \notag  \\
&\qquad\qquad \cup \{(x,y): 0\le x < \delta,~ f_2(x) < y < \ep \}, \label{eq:A}
\end{align}
where $f_1, f_2$ are $C^{1,\alpha}$-functions from $(-\delta,\delta)$ into $(-\ep,\ep)$ such that $f_1(0)=f_2(0)=0$.  
Let $n_1(\pp)$ be the limit of $n(x)$ at $x=\pp$ along a connected component of $\partial\AA\cap U$ and $n_2(\pp)$ the limit along the other component, and   
\begin{equation}\label{eq:n_corner}
n(\pp):=\frac{n_1(\pp)+n_2(\pp)}{\|n_1(\pp)+n_2(\pp)\|}. 
\end{equation}
Note that $n_1(\pp) + n_2(\pp)$ never vanishes, so that $n(\pp)$ is well defined. 
Then the cone $C$ in \eqref{eq:C} is exactly the connected component surrounded by the left and right tangent lines of $\AA$ at $\pp\in \partial \AA$ and containing the half-line $\{\pp -t n(\pp): t>0\}$, see Fig.~\ref{fig2}.

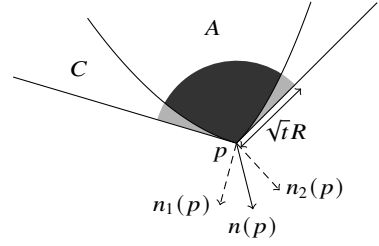
\begin{figure}[t]
\sidecaption[t]
\begin{tikzpicture}[scale =1.1]
\draw[domain=0:0.9] plot(\x, \x*\x+\x);
\draw[thin, domain=-1.8:0] plot(\x, 0.3*\x*\x-0.3*\x);  
\draw (0,0) -- (1.7,1.7);
\draw (0,0) -- (-2.7,0.8);
\draw node at (-0.3,1.4) {$\AA$};  
\draw node at (-1.9,0.9) {$C$};  
\draw node at (0,0.05) [below left] {$\pp$};  
\path[fill, opacity=0.8] plot[domain=0:0.54369] (\x, \x*\x+\x) arc [start angle = 57.2958, end angle = 150.688, radius = 1] plot[domain=-0.87192:0] (\x, 0.3*\x*\x-0.3*\x);
\path[fill, opacity=0.3] plot[domain=0:0.54369] ({\x}, {\x*\x+\x)}) arc [start angle = 57.2958, end angle = 45, radius = 1] plot[domain=0:0.707] (\x, \x);
\path[fill, opacity=0.3] plot[domain=0:-0.957826] (\x, -0.3*\x)  arc [start angle = 163.301, end angle = 150.688, radius = 1] plot[domain=-0.87192:0] (\x, 0.3*\x*\x-0.3*\x);
\draw[->] (0,0) -- (0.2,-0.8); 
\draw node at (0.2,-0.8) [below] {$n(\pp)$}; 
\draw[<->] (0.05,-0.05) -- (0.78,0.66); 
\draw node at (0.6,0.15)  {\small$\sqrt{t}R$};  
\draw[->, densely dashed] (0,0) -- (-0.2,-0.75); 
\draw node at (-0.2,-0.75) [left] {$n_1(\pp)$}; 
\draw[->, densely dashed] (0,0) -- (0.52,-0.57); 
\draw node at (0.5,-0.55) [right] {$n_2(\pp)$}; 
\end{tikzpicture}
\caption{Approximating $\AA$ with $C$ near $\pp\in \partial \AA$ in two dimensions. The black region is $U_t$ and the gray region is $U_t\Delta V_t$.}\label{fig2}
\end{figure}

Since  for sufficiently small $\sqrt{t}R$
\begin{align*}
U_t\Delta V_t 
&\subseteq  \{(x, s f_1(x)+(1-s) f_1'(0)x): -\delta < x  \le 0, 0 \le s \le 1 \}  \\ 
&\qquad\cup  \{(x, s f_2(x)+(1-s) f_2'(0)x): 0\le x < \delta,0 \le s \le 1 \}, 
\end{align*}
the corresponding argument in the proof of Theorem \ref{Thm:normal} yields 
\begin{equation} \label{eq:UV2}
|U_t\Delta V_t| \le  \int_{[-\sqrt{t}R, \sqrt{t}R]} \max_{1\le i \le 2}|f_i(x)-f_i'(0)x|\,\d x = O(t^\frac{\alpha+2}{2}R^{2+\alpha}), 
\end{equation}
so that condition \eqref{eq:UV} is satisfied for a suitable function $R(t)$. See also Fig.\ \ref{fig2}. 

\vspace{2mm}
\noindent
\textbf{Higher dimensions.}  Let $N\ge3$. By condition (C1) there are $C^{1,\alpha}$-functions $f_i$ with $f_i(0)=0$ such that $\AA-p$ is  written as 
\[
\left\{(z', z_N): z_N > \min_{1\le j \le m} \max_{i \in I_j} f_i(z') \right\}
\] 
in a neighborhood of 0 in an orthonormal basis. Correspondingly, the cone $C-\pp$ is of the form 
\[
\left\{(z', z_N): z_N > \min_{1\le j \le m} \max_{i \in I_j}\, (\nabla f_i(0),z') \right\}. 
\] 
For $i\ne j$, let $W_t^{i,j}\subseteq  \{z' \in \R^{N-1}: \|z'\|\le \sqrt{t}R\}$ be the region where $\partial\AA-\pp$ and $\partial C-\pp$ are expressed as the graphs of $f_i(z')$ and $(\nabla f_j(0), z')$, respectively. Note that $W_t^{i,j}$  is a region surrounded by the circle $\{z': \|z'\|=\sqrt{t}R\}$, the projection of the $C^{1,\alpha}$-submanifold $\{z: g_i(\pp+z)=g_j(\pp+z)=0\}$  onto the $z'$-plane and the projection of  the hyperplane $\{z: (\nabla g_i(\pp),z)= (\nabla g_j(\pp),z)\}$ onto the $z'$-plane.   Let $W_t:= \cup_{i\ne j}W_t^{i,j}.$  Similar to the proof of \eqref{eq:UV2} in two dimensions, one can show $|W_t| = O(t^\frac{\alpha+N-1}{2}R^{\alpha+N-1})$.  For sufficiently small $t>0$, the decomposition   
\begin{align*}
&U_t\Delta V_t \subseteq  \bigcup_{i\ne j} \{(z',  s f_i(z') +(1-s) (\nabla f_j(0),z')): z' \in W_t, 0\le s \le 1\}  \\ 
&\quad\quad\qquad \cup \bigcup_{i=1}^k \{( z', s f_i(z') +(1-s) (\nabla f_i(0),z')): z' \in \R^{N-1}\setminus W_t, 0\le s \le 1 \} 
\end{align*}
holds.  We obtain, by using the estimate $|f_i(z') - (\nabla f_j(0),z')| \le C_1\|z'\|$ in a neighborhood of $0$, 
\begin{align*}
|(U_t\Delta V_t)\cap (W_t \times \R )| 
&\le  \int_{W_t}\max_{i\ne j} | f_i(z') - (\nabla f_j(0),z') |\,\d z' \\
&\le C_1\sqrt{t}R|W_t|=  O(t^\frac{\alpha+N}{2}R^{\alpha+N}), 
\end{align*}
and also, by using the estimate $|f_i(z') -(\nabla f_i(0), z') | \le C_2 \|z'\|^{1+\alpha}$, 
\begin{align*}
|(U_t\Delta V_t)\cap ( (\R^{N-1}\setminus W_t)\times \R)| 
&\le  \int_{\substack{z'\in \R^{N-1}\setminus W_t\\ \|z'\|\le \sqrt{t}R}} 
\max_{1 \le i \le k}|f_i(z') -(\nabla f_i(0), z')|\,\d z' \\
&\le  C_2\int_{\|z'\|\le \sqrt{t}R}
\|z'\|^{1+\alpha}\,\d z' = O(t^\frac{\alpha+N}{2}R^{\alpha+N}), 
\end{align*}
so that $|U_t\Delta V_t| = O(t^\frac{\alpha+N}{2}R^{\alpha+N})$, i.e., condition \eqref{eq:UV} is fulfilled. 
\end{proof}


\subsubsection{An explicit formula in two dimensions}\label{sec:explicit}

At $C^{1,\alpha}$-corners $\pp$ in two dimensions, the main term $\int_{C-\pp} z \kappa'(\|z  \|^2) \,\d z$ of \eqref{eq:normal3}  can be more concretely calculated. 


\begin{proposition}\label{Thm:piecewise} 
Let $N=2$ and $\pp$ be a $C^{1,\alpha}$-corner of $\AA$. Let $n(\pp)$ be defined as \eqref{eq:n_corner}. We define
$$
\xi = \xi_\pp = 2\sin \frac{\phi_\pp}{2} \int_0^\infty \kappa(r^2)\,\d r
$$
 where $\phi_p\in (0,2\pi)$ is the inner angle of the curve $\partial \AA$ at $\pp$. Then formula \eqref{eq:asymptotic_normal} holds true at $x=\pp$ in the sense of pointwise convergence. 
 
 \end{proposition}
 \begin{remark}
Due to the factor $\sin \frac{\phi_\pp}{2}$ the length of $\nabla_x h(t,\pp)$ at a corner $\pp$ is shorter than at smooth points unless $\phi_p = \pi$, i.e., unless $\partial \AA$ is a $C^{1,\alpha}$-surface at $\pp$.  
 \end{remark}
\begin{proof}[Proof of Proposition \ref{Thm:piecewise}] 
To compute the integral $\int_{C-\pp} z \kappa'(\|z  \|^2) \,\d z$, let $z= - u n(\pp) + v n_{\perp}(\pp)$ be the change of variables to $(u,v)$, where $n_\perp(\pp)$ is a unit vector orthogonal to $n(\pp)$. Let $\widehat C=\{(u,v)\in \R^2\setminus\{0\}: \arg (u+i v) \in (-\phi_\pp/2,\phi_\pp/2)\}$ which is nothing but the shifted cone $C - \pp$ in the $(u,v)$-coordinate. By the symmetry of $\widehat C$ with respect to the $u$-axis and oddness of functions, 
\begin{eqnarray*}
\int_{\substack{\|z\|\leq R\\ z\in C -\pp}} z \kappa'(\|z\|^2) \,\d z
&=& \int_{\substack{\|(u,v)\|\leq R\\ (u,v)\in \widehat C}} (-u n(\pp) + v n_{\perp}(\pp)) \kappa'(u^2+v^2) \,\d u \d v \\
&=& -n(\pp) \int_{\substack{\|(u,v)\|\leq R\\ (u,v)\in \widehat C}} u \kappa'(u^2+v^2) \,\d u \d v \\
&=& -n(\pp) \int_{\widehat C} u \kappa'(u^2+v^2) \,\d u \d v +\widetilde \ep_3(t),  
\end{eqnarray*}
where $\widetilde \ep_3(t)$ satisfies the same bound \eqref{eq:E3} used for $\ep_3(t)$. 
To finish the calculations, going to the polar coordinate amounts to 
\begin{align}
-\int_{ \widehat C} u \kappa'(u^2+v^2) \,\d u \d v \notag
&= -\int_{0}^\infty \,\d r \int_{-\phi_\pp/2}^{\phi_\pp/2}\,\d\theta \cos \theta \kappa'(r^2) r^2 \notag \\ 
&= \sin \frac{\phi_\pp}{2} \int_0^\infty \kappa(r^2)\,\d r.     \tag*{\qedhere}
\end{align}
\end{proof}


\begin{note}
The results of Section \ref{sec3} are improved from \cite{M1}: Theorem \ref{Thm:normal} and Proposition \ref{Thm:piecewise} treated $C^{1,\alpha}$-surfaces and corners while the previous paper \cite{M1} focused on more restrictive $C^2$ cases; Theorems \ref{Thm:gen} and \ref{prop:corner} are new results that are not included in \cite{M1} and that can be applied to singular points in any dimensions. 
\end{note}

\section{Distance function}
\label{sec4}

Varadhan gave a method to extract the distance function from an elliptic equation with small diffusion coefficient~\cite{V67}.   
He actually discussed the case of variable coefficients but here we focus on the simplest case of Laplacian. 
Let $\AA \subseteq \R^N$ be a bounded open subset with suitable regularity, $a>0$ a constant and $q_a$ be a unique solution to the equation 
\begin{equation}\label{eq:Varadhan}
   \begin{cases}
-a\Delta q_a+q_a = 0 &\text{in}\quad \AA,\\
q_a=1&\text{on}\quad \partial \AA. 
\end{cases}
\end{equation}
Varadhan proved 
\begin{equation*}
-\sqrt{a} \log q_a(x) \to d(x, \partial \AA) \quad \text{uniformly for} \quad x \in \AA.
\end{equation*}
We generalized Varadhan's equation to \eqref{eq:sd0} in \cite{M}. 


\subsection{Results}

Throughout Section \ref{sec4},  the following are assumed. 

\vspace{4mm} 
\noindent
{\bf Assumption.}   
\begin{enumerate}
   \item $\Omega$ is a bounded $C^{1,1}$-domain of $\R^N$. 
  \item  $f\colon \Omega \to \R$ is a measurable bounded nonnegative function and the set $[f>0]:=\{x\in \Omega\colon f(x)>0\}$ is nonempty and open in $\Omega$.

    \item\label{item3} The set  
$\AA:=\Omega\setminus \overline{[f>0]}$  is nonempty and satisfies $d(\AA, \partial\Omega)>0$, where $d(\AA,\partial \Omega):=\inf_{x\in \AA, y  \in \partial \Omega}|x-y|$, and $\AA$ consists of finitely many connected components of $C^{1,1}$-boundary. See also Fig.~\ref{fig:concept1}. 

\begin{figure}[t]
\sidecaption[t]
       \includegraphics[scale = 0.28, bb=100 0 857 494]{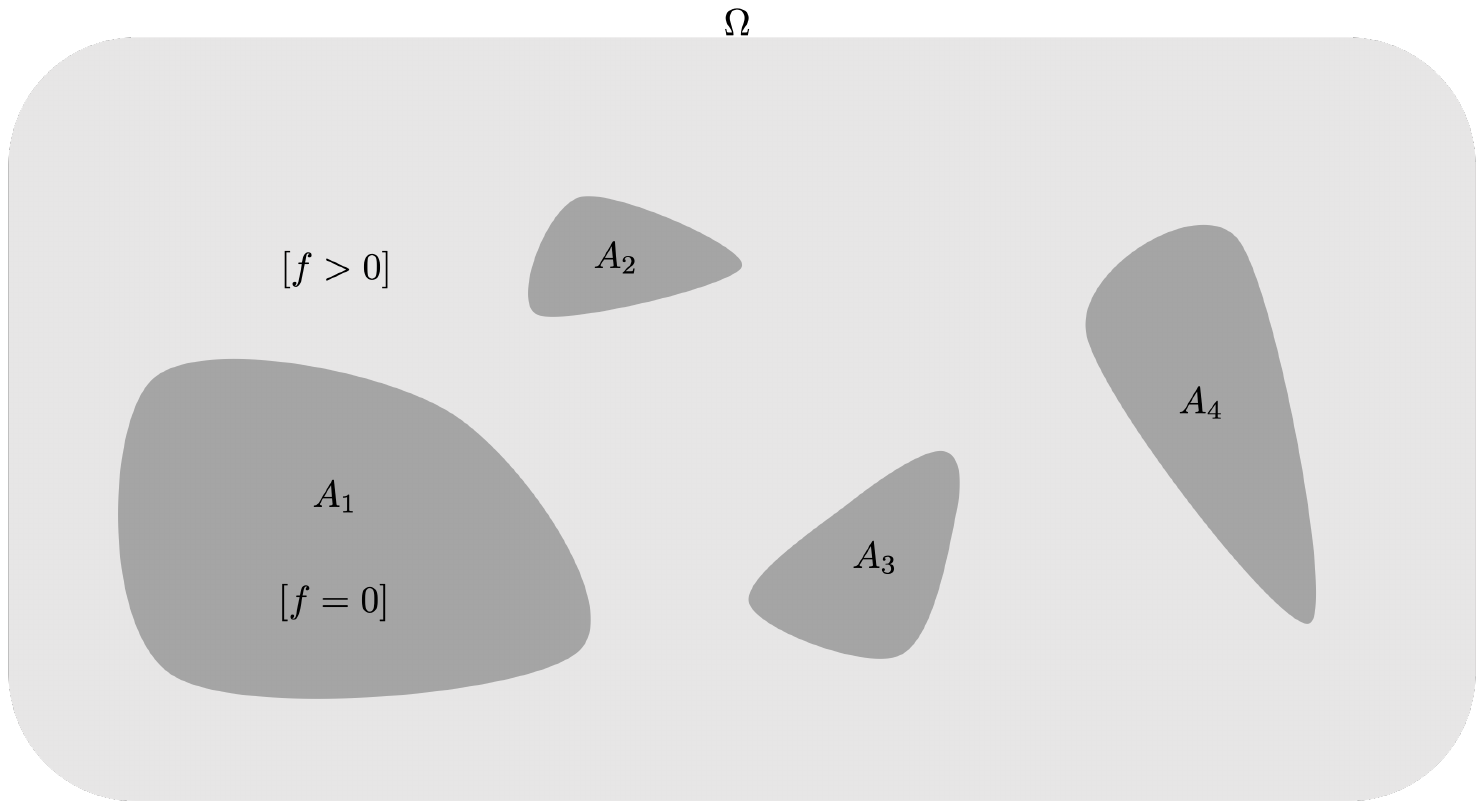} 
   \caption{A typical configuration of $\Omega$ and $\AA$ consisting of four connected components $\AA_1, \AA_2, \AA_3, \AA_4$.  }
\label{fig:concept1}
\vspace*{-10mm}
\end{figure}

    \item $g$ is continuous and nonnegative on $\partial\Omega$, and  moreover, there exists 
$\tilde{g}\in H^1(\Omega)\cap C(\overline{\Omega})$ such that $\tilde g |_{\partial \Omega} = g$.
\end{enumerate}

A weak solution to Equation \eqref{eq:sd0} can be formulated as a function $u\in H^1(\Omega)$ such that $u-\tilde{g}\in H^1_0(\Omega)$ and  
\begin{equation}\label{eq:sd_weak}
a \int_\Omega (\nabla u, \nabla \varphi)\, \d x+ \int_{\Omega} u\varphi\, \d x=\int_\Omega f\varphi\, \d x, \qquad \varphi\in H^1_0(\Omega).
\end{equation}
By Stampacchia's theorem (see, e.g.,~\cite[Theorem 5.6]{B11}) such a weak solution $u=u_a$ exists and is unique. 
Moreover, $u_a\in C(\overline{\Omega})$ (see  \cite[Corollary 9.18]{GT}) and $u_a$ is positive on $\Omega$ (see \eqref{eq:minmax}). 

\begin{theorem}[\cite{M}]\label{T:main}
It holds that, as $a\to0^+$, 
$$
-\sqrt{a}\log u_a(x) \to d (x, \partial \AA)
\quad \text{ uniformly on } \AA.   
$$   
\end{theorem}

A convergence rate can be obtained with a moderate additional assumption on the volume mean of $f$. Let $B_\ep(y)$ be the open ball of radius $\ep>0$ centered at $y \in \R^N$ and $\vmean_\ep^y(h)$ be the volume mean of a function $h$
$$
\vmean_\ep^y(h)
:=
\frac{1}{| B_\ep(y)|}
\int_{B_\ep(y)} h(x)\, \d x.
$$

\begin{theorem}[Rate of convergence \cite{M}]\label{T:rate}
Assume that there exists $0\le \zeta<\infty$ such that
\begin{equation} \label{eq:mean_f}
0<  \inf_{\substack{0< \ep <1\\ y \in \partial \AA}}\ep^{-\zeta k} \vmean_\ep^y(f^k) \le  \sup_{\substack{0< \ep <1\\ y \in \partial \AA}}\ep^{-\zeta k} \vmean_\ep^y(f^k) <\infty, \qquad  k\in\{1,2\}. 
\end{equation}
 Then for every $0<\tau <1/2$ there exists a constant $C= C(\Omega,\AA,f,g,\tau)>0$ such that
\begin{equation*}
\sup_{x\in \overline{\AA}} \left| -\sqrt{a} \log u_a(x) - d(x, \partial \AA)\right| \le Ca^{\frac1{2}-\tau} \quad \text{for all} \quad 0<a<\frac1{2}.  
\end{equation*}
Moreover, if $\zeta=0$ then the bound $C a^{\frac1{2}-\tau}$ can be improved to $\tilde C \sqrt{a} \log \frac{1}{a}$ for some $\tilde C=\tilde C(\Omega, \AA, f,g)>0$. The rate $\sqrt{a}\log(1/a)$ is optimal  at least in dimension one in the sense that there is an example of $\Omega$, $f$ and $g$ that precisely give this rate. 
\end{theorem}

Condition \eqref{eq:mean_f} means that the function $f$ should not be ``too smooth'' near $\partial \AA$. The parameter $\zeta$ stands for the smoothness.  For example, the characteristic function $f=\chi_{\Omega \setminus \overline{\AA}}$ satisfies \eqref{eq:mean_f} with $\zeta=0$ and the function $f(x)=(\max\{0, \|x\|^2-r^2 \})^\zeta$ satisfies \eqref{eq:mean_f} for any parameters $r,\zeta>0$ as long as $\overline{\AA} = \overline{B_r(0)}$ is contained in $\Omega$.

\begin{remark} Equation \eqref{eq:sd_weak} includes Equation \eqref{eq:Varadhan} as the special case $f\equiv 0$ and $g\equiv1$. To be precise, this case is excluded due to Assumption \ref{item3}; however, our proof can be easily fixed to treat this case and the proof becomes much simpler. We can actually thus show, under the assumption that $\AA$ consists of finitely many connected components of $C^{1,1}$-boundary, that 
\begin{equation*}
\sup_{x\in \overline{\AA}} \left| -\sqrt{a} \log q_a(x) - d(x, \partial \AA)\right| \le C\sqrt{a}\log\frac1{a} \quad \text{for all} \quad 0<a<\frac1{2}  
\end{equation*}
for some constant $C=C(\AA)>0$. 
\end{remark}

\begin{remark}
Similar problems are discussed in the context of the method of vanishing viscosity for viscosity solutions to Hamilton--Jacobi equations. 
First, the Cole-Hopf transform $p_a(x) = -\sqrt{a} \log u_a(x)$ transforms our equation \eqref{eq:sd0} to  
\begin{equation*}
\begin{cases}
-\sqrt{a}\Delta p_a+|\nabla p_a|^2-1 =-f\exp(\frac{p_a}{\sqrt{a}})& \text{in}\quad \Omega, \\
p_a = -\sqrt{a}\log g & \text{on} \quad  \partial \Omega. 
\end{cases}
\end{equation*}
Recalling that $f=0$ in $A$, we have   
 \begin{equation}\label{eq:eikonal}
-\sqrt{a}\Delta p_a+|\nabla p_a|^2-1 =0 \quad \text{in} \quad A. 
\end{equation}
Equations that generalize \eqref{eq:eikonal} are analyzed in \cite{FS86} and \cite{Tra11} with zero boundary condition on $\partial A$. More precisely,  an asymptotic expansion as $a\to0^+$ was obtained in \cite[Theorem 5.1]{FS86} and  the rate of convergence $a^{\frac1{4}}$ was obtained in \cite[Theorem 3.3]{Tra11}. In our problem, however, the  values of $p_a$ on $\partial A$ are unknown and the major part of our proof is devoted to the estimate of these boundary values. Consequently we obtained the better rate of convergence $a^{\frac{1}{2}-\tau}$ or $\sqrt{a}\log(1/a)$. 
We also mention that in \cite[Proposition 6.1]{L} (see also the remark after the proposition) the rate of convergence $a^{\frac1{4}}$ has also been obtained when the domain is the whole space $\R^N$.
\end{remark}

If $f$ is a characteristic function, then we can also analyze $u_a$ on $\Omega \setminus \AA$ and extract the signed distance function from the PDE. The proof is just a slight modification of Theorem \ref{T:rate}. 

\begin{theorem}[\cite{M}]
We select $f:= C^* \chi_{\Omega \setminus \overline{\AA}}$ with any constant $C^* \ge \sup_{x\in \partial\Omega}g(x)$ and define  
 $U_a\in H^1(\Omega)$ as in \eqref{eq:sdist}.  Then there exists $C= C(\Omega,\AA,C^*,g)>0$ such that 
\begin{equation*}\label{eq:rate_signed_distance}
\sup_{x \in \Omega^*} |U_a(x) - \hat{d}(x,\partial \AA)| \le C \sqrt{a} \log \frac1{a} \quad\text{for all} \quad 0< a< \frac{1}{2}, 
\end{equation*}
where $\hat{d}$ is the signed distance function 
\begin{equation*}
\hat{d}(x,\partial \AA) := 
\begin{cases}
-d(x,\partial \AA), & x \in \overline{\AA}, \\
d(x, \partial \AA), & x \in \Omega \setminus \overline{\AA} 
\end{cases}
\end{equation*}
and $\Omega^\ast:=\{ x\in \Omega\colon d(x,\partial \AA)\le d(x,\partial\Omega)\}$. 
\end{theorem}

\subsection{Sketch of proofs}

Because the original proof is lengthy, the proof is sketched below so that the idea can be grasped. The reader is referred to the original article \cite{M} for further details. 

\vspace{3mm}
\noindent
\textbf{[1. Positivity and uniform bound for $u_a$]\hspace{2mm}} It is known that
\begin{align}\label{eq:minmax}
0<u_a(x)\le M\quad \text{ for all } x\in \Omega, 
\end{align}
where $M:=\max\left\{\sup_{x\in \Omega} f(x),  \sup_{x\in \partial\Omega} g(x)\right\}$, see e.g.\ \cite[Theorem 9.27 and footnote 34]{B11}. A slightly different proof is given in \cite{M}. 

\vspace{3mm}
\noindent
\textbf{[2. Lower bound for $-\sqrt{a} \log u_a(x)$]\hspace{2mm}}  The proof follows the lines of Varadhan's \cite{V67}. Since $f=0$ on $\AA$, the function $u_a$ restricted to $\AA$ satisfies  
\[
-a\Delta u_a+u_a = 0 \quad \text{in} \quad \AA 
\] 
and so $u_a|_{\overline{\AA}} \in C^2 (\AA) \cap C(\overline{\AA})$. 
We fix $y \in \AA$, set $\ep:= d(y,\partial \AA)$ and consider a unique solution $w$ in $C^2 (B_\ep(y)) \cap C(\overline{B_\ep(y)})$ to the equation
\begin{equation}
   \begin{cases}
-a\Delta w+ w = 0 &\text{in} \quad B_\ep(y),\\
w=M&\text{on} \quad \partial B_\ep(y). 
\end{cases}
\end{equation}
Since $u_a \le w$ on $\partial B_\ep(y)$, the comparison principle yields $u_a \le w$ in $B_\ep(y)$.  The function $w$ has the explicit form
\begin{equation}\label{eq:w}
w(x)=
\begin{cases}
\displaystyle \frac{M\int_0^\pi \cosh(|x-y|\cos \theta/\sqrt{a})\sin^{N-2}\theta\, \d \theta}{\int_0^\pi \cosh(\ep\cos \theta/\sqrt{a})\sin^{N-2}\theta\, \d \theta}
&\text{if } N\ge 2, 
\vspace{3mm}\\
\displaystyle \frac{ M\cosh(|x-y|/\sqrt{a})}{ \cosh(\ep/\sqrt{a})}
&\text{if } N=1. 
\end{cases}
\end{equation}
 Some straightforward estimates of $w$ entail 
\begin{align} 
-\sqrt{a}\log u_a(y)-d(y,\partial \AA)
\ge 
c \sqrt{a} \log \frac1{a}, \quad y\in \AA,~ 0 < a <\frac1{2},     \label{eq:lower}
\end{align}
for some constant $c>0$ independent of $a$ and depending on the dimension $N$, the global bound $M$ and the diameter of $\AA$. 

\vspace{3mm}
\noindent
\textbf{[3. Upper bound for $-\sqrt{a} \log u_a(x)$]\hspace{2mm}}  Let 
\[
\beta_a :=\inf_{y\in\partial \AA}u_a(y).   
\]
Since $u_a$ is continuous function on $\overline{\AA}$ and is positive, the above infimum is attained at a point in $\partial \AA$ and $\beta_a$ is positive. 

By contrast to the lower bound, we now take a ball outside of $\AA$ and construct a comparison function. Let $y \in \AA$ be fixed and $\sigma \in \partial \AA$ be a point such that $|y-\sigma| =d(y, \partial \AA)$. For each $0<a\ll 1$, there exists a ball $B_{\sqrt{a}}(\tilde{y}) \subseteq \Omega\setminus \AA$ tangent to $\partial \AA$ at $\sigma$. Let  $z$ be a unique classical solution to 
\begin{equation*}
\begin{cases}
a\Delta z = z &\hspace{-15mm} \quad \text{in} \quad \R^N\setminus B_{\sqrt{a}}(\tilde{y}),\\
z=\beta_a&\hspace{-15mm} \quad \text{on}\quad \partial B_{\sqrt{a}}(\tilde{y}),\\
\displaystyle \lim_{|x|\to +{\infty}}z(x)=0. &
\end{cases}
\end{equation*} By the maximum principle, $z \le \beta_a$ on $\R^N\setminus B_{\sqrt{a}}(\tilde{y})$ and hence on $\AA$. 
Since now $u_a \ge \beta_a \ge z$ on $\partial \AA$, we get $u_a \ge z$ on $\AA$. Again, $z$ has an explicit form, which is omitted here.  
Estimating $z(y)$ we arrive at 
\begin{align} \label{eq:upper}
&-\sqrt{a}\log u_a(y)-d(y,\partial \AA)
\le    \sqrt{a}\log \frac1{\beta_a }  
+c' \sqrt{a}\log\frac1{a}  
\end{align}
for all $y \in \AA$ and $0<a <\frac1{2}$, where $c'>0$ is a constant depending only on the dimension $N$ and the diameter of $\AA$. 

The above arguments more or less followed the lines of \cite{V67} too. In the case of $q_a$, the boundary value was $1$ and so there was no need to study $\beta_a$.  In our case, however, we need to estimate $\beta_a$, which requires substantially new ideas. The next steps 4--7 are devoted to $\beta_a$. 

\vspace{3mm}
\noindent
\textbf{[4. Upper bound for $\sqrt{a} \log (1/\beta_a)$]\hspace{2mm}}
For each $a>0$, we choose a point $y_a \in \partial \AA$ such that $u_a(y_a)=\beta_a$, which exists by the continuity of $u_a$. 
We prove that, for any $0<\ep<d(\partial \AA, \partial \Omega):= \inf\{|x-y|:(x,y)\in\partial \AA \times \partial \Omega\}$ and $a>0$,  
\begin{align} 
\sqrt{a}\log \frac1{\beta_a}
\le
\ep-
\sqrt{a}\log \vmean_{\ep}^{y_a}({u_a}).  \label{eq:logC_a}
\end{align}

Again we begin with constructing a comparison function. For each  $0<\eta<  d(\partial \AA, \partial \Omega)$, let $u_a^\eta$ be a unique classical solution to
\begin{equation*}
\begin{cases}
a\Delta u_a^\eta = u_a^\eta & \quad \text{in}\quad B_\eta(y_a),\\
u_a^\eta={u}_a& \quad \text{on}\quad \partial B_\eta(y_a).
\end{cases}
\end{equation*}
Now the boundary values are not constant and so an explicit form of $u_a^\eta$ is hopeless. Still, the value of $u_a^\eta$ at the central point $y_a$ is explicit as shown in  \cite{Kuz}. Actually, the formula for $u_a^\eta(y_a)$ is  exactly formula \eqref{eq:w} evaluated at $x=y=y_a$ by replacing $M$ with the sphere mean of  $u_a$ over $\partial B_\eta(y_a)$, denoted by $\smean_\eta^{y_a}({u}_a)$. After some estimates, we can obtain 
\begin{align}\label{eq:uaeta}
u_a^\eta(y_a) \ge \frac{\smean_\eta^{y_a}({u}_a)}{ \cosh(\eta/\sqrt{a})}.   
\end{align}   
As in step 1 above, applying \cite[Theorem 9.27]{B11} to the function $u_a - u_a^\eta\in H_0^1(B_\eta(y))$ yields $u_a \ge u_a^\eta$ in $B_\eta(y)$. Combining \eqref{eq:uaeta} with $u_a(y_a)\ge u_a^\eta(y_a)$, applying $\int_0^\ep \cdot\, \eta^{N-1}\,\d \eta$, and recalling that the volume mean and sphere mean are related via
\begin{equation*}
    \vmean_\ep^y(h) 
    = \frac{N}{\ep^N} \int_0^\ep     \smean_\eta^y(h)\eta^{N-1}\, \d \eta,    
\end{equation*}
we obtain \eqref{eq:logC_a}.

\vspace{3mm}
\noindent
\textbf{[5. Lower bound for $\vmean_{\ep}^{y_a}({u_a})$]\hspace{2mm}}  This part is proven via the weak form of $u_a$. The idea is that  if we could take the test function $\varphi$ in \eqref{eq:sd_weak} to be the characteristic function $\chi_{B_\ep(y_a)}$, then the second term would yield the volume mean of $u_a$. Of course the characteristic function is not in $H_0^1(\Omega)$, but approximation via mollifiers works. We can neglect the first term of \eqref{eq:sd_weak} thanks to the small coefficient $a$. Consequently, we obtain 
\begin{equation}\label{eq:f}
\sup_{y\in \partial \AA} |\vmean_\ep^y(u_a) - \vmean_\ep^y(f)| \to 0 \quad \text{as}\quad a \to0^+. 
\end{equation}
Since $\partial \AA$ is the boundary of the open set $[f>0]$,  the volume mean $\vmean_\ep^y(f)$ is positive for all $y \in \partial \AA$. By the continuity of $y\mapsto \vmean_\ep^y(f)$, it has a positive minimum on $\partial \AA$. Combined with \eqref{eq:f}, this gives a uniform lower bound $\vmean_\ep^y(u_a) \ge \alpha(\ep) >0$ for some $\alpha(\ep)>0$ independent of $a$. 

\vspace{3mm}
\noindent
\textbf{[6. Conclusion: uniform convergence]\hspace{2mm}}
Considering \eqref{eq:logC_a} and the lower bound $\vmean_{\ep}^{y_a}({u_a})\ge \alpha(\ep)$, we obtain 
\[
\limsup_{a\to0^+} \sqrt{a} \log \frac1{\beta_a} \le 0.
\] This fact combined with inequalities \eqref{eq:lower} and \eqref{eq:upper} finishes the proof of Theorem \ref{T:main}.

\vspace{3mm}
\noindent
\textbf{[7. Rate of convergence]\hspace{2mm}}  In view of \eqref{eq:lower} and \eqref{eq:upper}, a good upper bound for $\sqrt{a}\log(1/\beta_a)$ is desirable.  For this, we choose $\ep = a^{\frac1{2}-\tau}$ in \eqref{eq:logC_a} as a function of $a$ ($0<\tau<\frac1{2}$ is arbitrarily small but fixed)  and give a bound for $\vmean_\ep^{y_a}(u_a)$. This can be achieved by carefully estimating the convergence \eqref{eq:f}. In the estimation we use the Schwarz inequality where the norms $\|u_a\|_{L^2(B_\ep(y))}, \|\nabla u_a\|_{L^2(B_\ep(y))} ~(y\in \partial \AA)$ appear. The mere use of the uniform bound $|u_a(x)|\le M$ results in $\|u_a\|_{L^2(B_\ep(y))}^2 \le C\epsilon^N $. This bound is kind of optimal if $\zeta=0$, e.g., in case $f=\chi_{\Omega\setminus \overline{\AA}}$. However,  the function $u_a$ is close to $f$ near $\partial \AA$ (see \eqref{eq:f}) and therefore, in case $\zeta>0$, the values of $u_a$ are likely to be close to $0$ near $\partial \AA$, so that $\|u_a\|_{L^2(B_\ep(y))}^2 \le C\epsilon^N $ might be far from optimal. In fact, by an iterative argument of local energy estimates, we can prove
\[
\|u_a\|_{L^2(B_\ep(y))}^2 + a \|\nabla u_a\|_{L^2(B_\ep(y))}^2 \le C \ep^{N + 2\zeta}, \qquad 0<a<1, y \in \partial \AA  
\]
for some uniform constant $C$ independent of $y \in \partial \AA$ and $a$. With this estimate, the convergence \eqref{eq:f} can be finely estimated and then Theorem \ref{T:rate} follows.

\begin{acknowledgement}
The authors would like to thank the anonymous referees for their careful reading and constructive suggestions, which have improved the clarity and readability of the manuscript.
\end{acknowledgement}
\ethics{Competing Interests}{\newline
This work was supported by JSPS Grant-in-Aid for Transformative Research Areas (B) grant
no. 23H03800JSPS. \newline
The authors have no conflicts of interest to declare that are relevant to the content of this chapter.}

\eject


\begin{thebibliography}{99.}%




\bibitem{AJT04} Allaire, G.,  Jouve, F., Toader, A. M.: 
Structural optimization using sensitivity analysis
and a level-set method. Journal of Computational Physics \textbf{194},  363--393 (2004)


\bibitem{BF} Belyaev, A. G., Fayolle, P.-A.: On variational and PDE-based distance function approximations.  
Computer Graphics Forum \textbf{34}, No. 8, 104--118   (2015)


 \bibitem{B11} Brezis, H.:
       Functional Analysis, Sobolev Spaces and Partial Differential Equations. Springer, New York (2011)

\bibitem{crane2013geodesics}
	Crane, K., Weischedel, C., Wardetzky, M.: 
	Geodesics in heat: A new approach to computing distance based on heat flow. 
	ACM Transactions on Graphics {\textbf{32}}, Issue 5, 1--11 (2013)
	
	\bibitem{FS86} Fleming, W. H., Souganidis, P. E.: Asymptotic series and the method of vanishing viscosity,  Indiana Univ. Math. J. {\bf 35}, No. 2, 425--447 (1986) 

   \bibitem{GT} Gilbarg, D.,  Trudinger, N. S.: Elliptic Partial Differential Equations of Second Order.  Classics in Mathematics. Springer-Verlag, Berlin (2001) 


        
\bibitem{han2004cortical}
	Han, X., Pham, D. L., Tosun, D., Rettmann, M. E., Xu, C., Prince, J. L.: 
	CRUISE: Cortical reconstruction using implicit surface evolution. 
	Neuro Image {\textbf{23}}, 3, 997--1012 (2004)


\bibitem{M1} Hasebe, T.,  Kuroda, H., Teramoto, H.,  Masamune, J., Yamada, T.: Construction of normal vector field using the partial differential equations (in Japanese). Transactions of the Japan Society for Industrial and Applied Mathematics \textbf{30}, 249--258 (2020) 

\bibitem{M} Hasebe, T.,  Masamune, J.,  Oka, T., Sakai, K., Yamada, T.:  Construction of signed distance functions with an elliptic equation. arXiv:2401.17665 (2024) 

\bibitem{JJT}  Jongen, H. T., Jonker, P., Twilt, F.:  Nonlinear Optimization in Finite Dimensions: Morse Theory, Chebyshev Approximation, Transversality, Flows, Parametric Aspects. Springer Science+Business Media, Dordrecht  (2000)


\bibitem{Kuz} Kuznetsov, N.: Mean value properties of solutions to the Helmholtz and modified Helmholtz equations. Journal of Mathematical Sciences {\bf 257}, 673--683  (2021) 

\bibitem{Leo} Leoni, G.: A First Course in Sobolev Spaces (Second Edition). Graduate Studies in Math. {\bf 181}. Amer. Math. Soc., Providence, RI (2017) 


 \bibitem{L} Lions, J. L.: Generalized Solutions of Hamilton-Jacobi Equations. Chapman and Hall/CRC Research Notes in Mathematics,
        Pitman Publishing (1982)
 
\bibitem{Mi} Michor, P. W.: Manifolds of Differentiable Mappings. Shiva Publishing Limited, Orpington (1980)

\bibitem{Mo} Morgan, F.: Geometric Measure Theory: A Beginner's Guide, fifth edition, Academic Press (2016)

\bibitem{NY} Nakayasu, A., Yamada, T.: Mathematical analysis of a partial differential equation system on the thickness. arXiv:2409.19958 (2024)

\bibitem{P} Peterson, D. W.: A review of constraint qualifications in finite-dimensional spaces.  SIAM Review {\bf 15}, No. 3, 639--654 (1973)


\bibitem{Tra11} Tran, H. V.: Adjoint methods for static Hamilton--Jacobi equations. Calc. Var. {\bf 41}, 301--319 (2011) 

\bibitem{V67}  Varadhan, S. R. S.:  
  On the behavior of the fundamental solution of the heat equation with variable coefficients. 
        Comm. Pure Appl. Math. \textbf{20}, 431--455  (1967)  



\bibitem{yamada2019geometric}
	Yamada, T.: 
	Geometric shape features extraction using a steady state partial differential equation system. 
	Journal of Computational Design and Engineering \textbf{6}, Issue 4, 647--656 (2019)


\bibitem{yamada2019thickness}
	Yamada, T.: 
	Thickness constraints for topology optimization using the fictitious physical model. 
	EngOpt 2018 Proceedings of the 6th International Conference on Engineering Optimization, 
	 483--490 (2019) 


\end{thebibliography}
\end{document}